\documentclass[12pt, reqno]{amsart}
\usepackage{amssymb,latexsym,amsmath,amscd,amsthm,graphicx, color}
\usepackage[all]{xy}
\usepackage{pgf,tikz}
\usepackage{mathrsfs}
\usetikzlibrary{arrows}
\usepackage[left=0.6 in, top=0.6 in, right=0.6 in, bottom=0.3 in]{geometry}
﻿
\definecolor{uuuuuu}{rgb}{0.26666666666666666,0.26666666666666666,0.26666666666666666}
\definecolor{qqwuqq}{rgb}{0.,0.39215686274509803,0.}
\definecolor{zzttqq}{rgb}{0.6,0.2,0.}
\definecolor{xdxdff}{rgb}{0.49019607843137253,0.49019607843137253,1.}
\definecolor{qqqqff}{rgb}{0.,0.,1.}
\definecolor{cqcqcq}{rgb}{0.7529411764705882,0.7529411764705882,0.7529411764705882}
﻿
﻿
﻿
﻿
﻿
﻿
﻿
﻿
\theoremstyle{plain}

\newtheorem{theorem}[subsection]{Theorem}

\newtheorem{lemma}[subsection]{Lemma}

\newtheorem{prop}[subsection]{Proposition}

\theoremstyle{definition}
\newtheorem{deli}[subsection]{Delineation}
\newtheorem{defi}[subsubsection]{Definition}

\newtheorem{example}[subsection]{Example}

\newtheorem{remark}[subsection]{Remark}

\newtheorem{note}[subsection]{Note}

﻿
﻿
﻿
﻿
﻿
\newcommand{\uu}{\cup}
﻿
\newcommand{\ci}{\subseteq}
\newcommand{\sci}{\subset}
\newcommand{\es}{\emptyset}
\newcommand{\set}[1]{\{#1\}}
﻿
﻿
\newcommand{\ga}{\alpha}
\newcommand{\gb}{\beta}

\newcommand{\gk}{\kappa}

\newcommand{\gn}{\nu}
\newcommand{\go}{\omega}

\newcommand{\gt}{\tau}

﻿
\newcommand{\gG}{\Gamma}
\newcommand{\gD}{\Delta}

﻿
\newcommand{\tit}{\textit}
﻿
\newcommand{\D}[1]{\mathbb{#1}}
﻿
\newcommand{\te}{\text}

\newcommand{\nd}{\noindent}

\begin{document}

\nd To appear, Chaos, Solitons and Fractals
	
	\title{Quantization dimension for a generalized inhomogeneous bi-Lipschitz iterated function system}

	\author{ Shivam Dubey}
	\address{Department of Applied Sciences\\
		Indian Institute of Information Technology Allahabad\\
		Devghat Jhalwa prayagraj}
	
	\email{rss2022509@iiita.ac.in}
	\author{ Mrinal Kanti Roychowdhury}
	\address{School of Mathematical and Statistical Sciences\\
		University of Texas Rio Grande Valley\\
		1201 West University Drive\\
		Edinburg, TX 78539-2999, USA.}
	\email{mrinal.roychowdhury@utrgv.edu}
	\author{Saurabh Verma}
	\address{Department of Applied Sciences\\
		Indian Institute of Information Technology Allahabad\\
		Devghat Jhalwa prayagraj}
	\email{Saurabhverma@iiita.ac.in}


	\date{}
	\maketitle
	
	\pagestyle{myheadings}\markboth{S. Dubey, M.K. Roychowdhury, and S. Verma }{Quantization dimension for a generalized inhomogeneous bi-Lipschitz IFS}
	
	\begin{abstract} 
For a given $r\in (0, +\infty)$, the quantization dimension of order $r$, if it exists, denoted by $D_r(\mu)$, of a Borel probability measure $\mu$ on ${\mathbb R}^d$ represents the speed how fast the $n$th quantization error of order $r$ approaches to zero as the number of elements $n$ in an optimal set of $n$-means for $\mu$ tends to infinity. If $D_r(\mu)$ does not exists, we call $\underline D_r(\mu)$ and $\overline D_r(\mu)$, the lower and upper quantization dimensions of $\mu$ of order $r$. In this paper, we estimate the quantization dimension of condensation measures associated with condensation systems $(\{f_i\}_{i=1}^N, (p_i)_{i=0}^N, \nu)$, where the mappings  $f_i$ are bi-Lipschitz and the measure $\nu$ is an image measure of an ergodic measure with bounded distortion supported on a conformal set. In addition, we determine the optimal quantization for an infinite discrete distribution, and give an example which shows that the quantization dimension of a Borel probability measure can be positive with zero quantization coefficient.  
	\end{abstract}
	\section{introduction} 
	Quantization refers to the process of converting a continuous range of values, such as real numbers, into a discrete set with fewer values. It has broad applications in various fields, including signal processing and digital communication (see \cite{GG, GKL, GN, Z}). For some recent results relevant to the work in this paper one can see \cite{DL, GL1, GL2, PRV2, PRV3, R1, R2, R3, R4, R5, R6, Z1, Z2, Z3}. Let $\mu$ be a Borel probability measure on $\mathbb R^d$, where $d \in \D N$. For $r \in (0,+\infty)$ and $n \in \mathbb N$, the \tit{$n$th quantization error} of order $r$ for $\mu$ is defined as
	\begin{equation}\label{eq1}
		V_{n,r} := V_{n,r}(\mu) = \inf \left\{ \int \rho(x, A)^r \, d \mu(x) : A \subset \mathbb R^d, \, 1\leq \text{card}(A) \leq n \right\},
	\end{equation} 
	where $\rho(x, A)$ represents the distance between the point $x$ and the set $A$, measured with respect to a given norm $\| \cdot \|$ on $\mathbb R^d$.
	﻿For a given $n\in \D N$, if there exists a set $A \subset \mathbb R^d$ with $ \text{card}(A) \le n$ such that the infimum in Equation~\eqref{eq1} is attained, then  $A$ is called an \tit{optimal set of $n$-means} for the probability measure  $\mu$. Let $\mu$ be a Borel probability measure on $\mathbb R^d$, and let $V_{n,r}(\mu)$ be the $n$th quantization error of order $r$. The numbers
	\begin{equation}\label{eq4}
		\underline{D}_r(\mu) := \liminf_{n \to \infty} \frac{r \log n}{-\log V_{n,r}(\mu)} \quad \text{and} \quad \overline{D}_r(\mu) := \limsup_{n \to \infty} \frac{r \log n}{-\log V_{n,r}(\mu)},
	\end{equation} 
	are called the \tit{lower} and \tit{upper quantization dimensions} of $\mu$ of order $r$, respectively. If both values are equal, the common value is referred to as the \tit{quantization dimension} $D_r(\mu)$ of order $r$ of the probability measure $\mu$. The quantization dimension measures the speed at which the quantization error for a measure approaches to zero as $n$ tends to infinity. For any $\gk > 0$, the \tit{ $\gk$-dimensional lower} and \tit{upper quantization coefficients} of order $r$ for a measure $\mu$ are given by
	\[
	\liminf_{n \to \infty} n^{\frac{r}{\gk}} V_{n,r}(\mu) \quad \text{and} \quad \limsup_{n \to \infty} n^{\frac{r}{\gk}} V_{n,r}(\mu),
	\]
	respectively. 
	\begin{remark}\label{rem24}(see \cite{GL2})
		If the $\gk$-dimensional lower and upper quantization coefficients of order $r$ for a measure $\mu$ are finite and positive, then the quantization dimension of $\mu$ of order $r$ exists and is equal to $\gk$.
	\end{remark} 
	﻿Let $\mathcal{I} := \{\mathbb{R}^d; f_1, f_2, \ldots, f_N\}$ be an iterated function system (IFS), where for each $i \in \{1, 2, \ldots, N\}$, there exists $0 < c_i < 1$ such that
	\[
	\rho(f_i(x), f_i(y)) \leq c_i \rho(x, y) \quad \forall x, y \in \mathbb{R}^d.
	\]
	Then, there exists a unique non-empty compact set $K_\es \subset \mathbb{R}^d$ called \tit{invariant set} or \tit{attractor} such that
	\[
	K_\es = \bigcup_{i=1}^N f_i(K_\es).
	\]
	Let $(p_1, p_2, \ldots, p_N)$ be a probability vector. Then, there exists a unique measure $\mu$ called \tit{invariant measure}, which has support the \tit{invariant set} $K_\es$, such that
	\[
	\mu = \sum_{i=1}^N p_i \mu \circ f_i^{-1}.
	\]
	Let $C \subset \mathbb{R}^d$ be a compact set. Then, there exists a unique non-empty compact set $K \subset \mathbb{R}^d$ called \tit{inhomogeneous invariant set} or \tit{inhomogneous attractor} such that
	\begin{equation}\label{eq2}
		K = \bigcup_{i=1}^N f_i(K) \cup C.
	\end{equation}
	Let $(p_0, p_1, p_2, \ldots, p_N)$ be a probability vector, and let $\nu$ be a Borel probability measure supported on $C$. Then, there exists a unique Borel probability measure $\mu$, which has support the \tit{inhomogeneous invariant set} $K$, such that
	\begin{equation}\label{eq3}
		\mu = \sum_{i=1}^N p_i \mu \circ f_i^{-1} + p_0 \nu.
	\end{equation}
For detailed results on the existence and uniqueness of inhomogeneous attractors and measures, see \cite{BD1,B}.
	We call $(\{f_i\}_{i=1}^N, (p_i)_{i=0}^N, \nu)$ a \tit{condensation system}. The measure $\mu$ is called the \tit{condensation measure} for the condensation system $(\{f_i\}_{i=1}^N, (p_i)_{i=0}^N, \nu)$. We say that the condensation system satisfies the \tit{strong separation condition} (SSC) if
	\[
	f_i(K) \cap f_j(K) = \emptyset \quad \forall~1 \leq i \neq j \leq N \quad \text{and} \quad f_i(K) \cap C = \emptyset \quad \forall~1 \leq i \leq N.
	\]
	We say that an IFS $\mathcal{I} = \{\mathbb{R}^d; f_1, f_2, \ldots, f_N\}$ satisfies the \tit{open set condition} (OSC) if there exists a bounded nonempty open set $U \subset \mathbb{R}^d$ such that
	\[
	\bigcup_{i=1}^N f_i(U) \subseteq U \quad \text{and} \quad f_i(U) \cap f_j(U) = \emptyset \quad \forall~1 \leq i \neq j \leq N.
	\]
	Moreover, the IFS satisfies the \tit{strong open set condition} (SOSC) if the open set $U$ can be chosen in such a way that it has a non-empty intersection with the attractor $K_\es$ of the IFS $\mathcal{I}$. Further, we say that the condensation system satisfies the \tit{condensation open set condition}, if the open set $U$ can be chosen in such a way that corresponding IFS satisfies the OSC and $C \subseteq \overline{U}$. In \cite{S}, Schief proved that the OSC and SOSC are equivalent for finite IFSs with similarity mappings in the Euclidean space $\mathbb{R}^d.$ However, there exist infinite IFSs that satisfy OSC but do not satisfy the SOSC; for such IFSs see \cite{Sz}. However, throughout this paper we assume that the condensation system satisfies the SSC. For some details about an inhomogeneous attractor and a condensation measure, one can see \cite{Sn, F, OS1, OS2, L2}.
	﻿	\begin{remark}
		We know that the classification of an IFS depends on the nature of its associated mappings. For instance, if the associated mappings are bi-Lipschitz, affine, or conformal, then the IFS is referred to as a bi-Lipschitz IFS, affine IFS, or conformal IFS, respectively. It is important to note that a bi-Lipschitz IFS can encompass all well-known classes of fractal sets, including self-similar, self-affine, and self-conformal sets.
		\begin{enumerate}
			\item If $ a_i = b_i $ in Equation \eqref{eq24}, then the IFS $\mathcal{I}$ reduces to a self-similar IFS, and the associated fractal set is referred to as a self-similar set.
			\item Let $ f : \mathbb{R}^d \to \mathbb{R}^d $ be an affine transformation given by $ f(x) = Ax + b $, where $ A $ is an injective linear transformation on $\mathbb{R}^d$. Then,
			\[
			\frac{1}{\|A^{-1}\|} \|x - y\|_2 \le \|Ax - Ay\|_2 = \|f(x) - f(y)\|_2 \le \|A\| \|x - y\|_2.
			\]
			\item Let $ X \subset \mathbb{R}^d $ be a nonempty compact set with $ X = \overline{\text{int}(X)} $, and let $\{(X, d); g_1, g_2, \ldots, g_M\}$ be a conformal IFS. Then, by \cite[Lemma 2.2]{P1}, there exists an open set $ V $ containing $ X $ and a constant $ c \geq 1 $, such that
			\begin{equation}\label{eq2000}
				c^{-1} \|g_{\omega}'\| \|x - y\| \leq \|g_{\omega}(x) - g_{\omega}(y)\| \leq c \|g_{\omega}'\| \|x - y\|,
			\end{equation}	for all $ x, y \in V $ and $ \omega \in \cup_{n \in \mathbb{N}} \{1, 2, \ldots, M\}^n $.
		\end{enumerate}
		From the above points it is clear that dimensional results applicable to bi-Lipschitz IFSs can provide useful estimates for the dimensions of self-similar, self-affine, and self-conformal IFSs. Nevertheless, it should be emphasized that bi-Lipschitz IFSs may not always yield the exact dimensional values for self-affine and self-conformal sets; see, for instance, \cite{F,LM}.
	\end{remark}
	\begin{note}
		Throughout the paper, we will consider the IFS $\mathcal{I}:=(\mathbb{R}^d; \{f_i\}_{i=1}^N, a_i, b_i)$ consisting of bi-Lipschitz mappings such that for all $x, y \in \mathbb{R}^d$,
		\begin{equation}\label{eq24}
			a_i \, \rho(x, y) \leq \rho(f_i(x), f_i(y)) \leq b_i \, \rho(x, y),
		\end{equation}
		where $0 < a_i \leq b_i < 1$.  Let  $\mathcal{I}_1 := (\mathbb{R}^d; \{g_j\}_{j=1}^M)$ be a conformal IFS with self-conformal set $C$, such that $\|g_j'\| \leq r$ for some $0 < r < 1$, where $\|g_j'\|$ represents the sup norm on the derivative map of $g_j$. For details about a conformal IFS and self conformal set, one can see, \cite{P1}. 
		Let $\sum_1 := \{1, 2, \ldots, M\}^{\mathbb{N}}$ be the collection of all sequences, also called infinite strings, in symbols $1, 2, \ldots, M$. For any $n \in \mathbb{N}$, $\sum_1^n:=\set{\go_1\go_2\cdots\go_n : \go_j\in \set{1, 2, \cdots, M} \te{ for all } 1\leq j\leq n}$ represents the collection of all finite strings of length $n$, and $\sum_1^* = \bigcup_{n \in \mathbb{N}\uu \set{0}} \sum_1^n$ is the collection of all finite strings, including the empty string $\emptyset$ of length zero. For any two strings $\go:=\go_1 \go_2\cdots \go_m$ and $\gt:=\gt_1\gt_2\cdots\gt_n$ in $\sum_1^\ast$, by $\go\gt:=\go_1\go_2\cdots\go_m\gt_1\gt_2\cdots\gt_n$, it is meant the concatenation of the two strings $\go$ and $\gt$.
		There is a natural connection between the code space and the invariant set given by the coding map $\pi : \sum_1 \to C$, defined by
		\[
		\pi(\omega) = \lim_{n \to \infty} g_{\omega_1 \circ \omega_2 \circ \cdots \circ \omega_n}(x).
		\]
		It is well-known that $\pi(\sum_1) = C$.
		Assume that $\hat\gn$ is a shift-invariant ergodic measure on $\sum_1$ satisfying the bounded distortion property, i.e., there exists a constant $L\geq 1,$ such that 
		\begin{equation} \label{D1} L^{-1}\hat \gn([\go])\hat\gn ([\gt])\leq \hat\gn([\go\gt])\leq L\hat \gn([\go])\hat\gn ([\gt]),
		\end{equation} 
		for  $\go, \gt \in \sum_1^\ast$, where for a finite string $u:=u_1u_2\cdots u_n\in \sum_1^\ast$ by $[u]$ it is meant the cylinder set $[u_1u_2\cdots u_n]=\set{x=(x_i)_{i=1}^\infty \in \sum_1 : x_i=u_i \te{ for } 1\leq i\leq n}$.  Let $\gn$ be the image measure of $\hat \nu$ under a coding map $\pi$ from $\sum_1$ to $C$, i.e., for any $\go \in \sum_1^\ast$, we have
		\[\nu(C_\go)=\hat \nu (\pi^{-1}(C_\go))=\hat \nu[\go],\]
		for more details see \cite{R5}. 
	\end{note}

	\begin{deli}
		﻿In this paper, we estimate the quantization dimension of a condensation system given by  $(\{f_i\}_{i=1}^N, (p_i)_{i=0}^N, \nu)$ generated by a bi-Lipschitz IFS, where the measure $\nu$ is the image measure of an ergodic measure with bounded distortion on the symbolic space  $\set{1, 2, \cdots, M}^{\D N}$ with support a conformal set $C$. We state and prove Theorem~\ref{thm1}, which is the main result of the paper. In addition, in Section~\ref{sec4}, we determine the optimal sets of $n$-means and the $n$th quantization errors for all $n\in \D N$ for an infinite discrete distribution $\gn$. For this $\gn$, in Theorem~\ref{theo2} we have shown that the quantization dimension $D(\gn)$ exists and equals zero, but the quantization coefficient does not exist. On the other hand, in Example~\ref{exam23}, we have shown that there is an infinite discrete distribution for which the quantization dimension exists, but the quantization coefficient is zero.  This examples shows that the converse of Remark \ref{rem24} does not hold in general.
	\end{deli}  
	
	﻿\begin{remark} \label{Duby1}
		In \cite{PRV}, Priyadarshi et al. estimated the quantization dimension of a  condensation measure $\mu$ associated with a condensation system $(\set{f_i}_{i=1}^N, (p_i)_{i=0}^N, \gn)$, where 
		$(\D R^d;  \set{f_i}_{i=1}^N)$ is a bi-Lipschitz IFS and the measure $\gn$ is generated by a contractive IFS $(\D R^d; g_1, g_2, \cdots, g_M)$ associated with a probability vector $(t_1, t_2, \cdots, t_M)$, i.e., $\gn=\sum_{i=1}^Mt_i\gn\circ g_i^{-1}$. Notice that if $C$ is the attractor of the IFS $(\D R^d; g_1, g_2, \cdots, g_M)$, then for a string $\go:=\go_1\go_2\cdots \go_n\in \sum_1^\ast$, we have 
		\[\gn(C_\go)=\gn(g_{\go_1} g_{\go_2}\cdots g_{\go_n})(C)=t_{\go_1}t_{\go_2}\cdots t_{\go_n}=t_\go.\]
	    We see that for any two strings $\go, \gt\in \sum_1^\ast$, we have 
		\[\gn(C_{\go\gt})=t_{\go\gt}=t_{\go} t_{\gt}=\gn(C_{\go})\gn(C_\gt),\]
		i.e., the measure $\gn$ considered in \cite{PRV} is a special case of an ergodic measure with bounded distortion, which is obtained by taking $L=1$ in the expression \eqref{D1}. Thus, the work in our paper, somehow generalizes the work in \cite{PRV}. 
	\end{remark}
	\section{preliminaries}
	To distinguish the symbolic space $\sum_1$ defined on the symbols $1,2,\ldots,M$ in the previous section, we define $\sum = \{1, 2, \ldots, N\}^{\mathbb{N}}$, as the collection of all sequences in symbols $1, 2, \ldots, N$. For any $n \in \mathbb{N}$, $\sum^n$ represents the collection of all finite strings of length $n$, and $\sum^* = \bigcup_{n \in \mathbb{N}\uu \set{0}} \sum^n$ is the collection of all finite strings including the empty string $\emptyset$ of length zero. For any string $\omega = \omega_1 \omega_2 \cdots\omega_n \in \sum^n$, we write $|\omega| = n$ to denote the length $n$ of $\omega$, and for any $k \leq n$, we define $\omega_{|_k} = \omega_1 \ldots \omega_k$ as the truncation of $\omega$ to the length $k$. For the juxtaposition of the strings $\omega = \omega_1 \omega_2 \cdots\omega_n \in \sum^n$ and $\tau = \tau_1 \tau_2 \cdots \in \sum$, we write $\omega \tau := \omega * \tau = \omega_1 \omega_2 \cdots\omega_n \gt_1\gt_2\cdots$ and say that $\tau$ is an extension of the string $\omega$, and write $\omega \prec \tau$, if $\omega = \tau_{|_{|\omega|}}$, i.e., if $\omega$ is the prefix of $\tau$. In this way, we define an ordering on $\sum^*$ as follows 
	\[\omega \prec \tau ~\text{whenever}~ |\omega| \le |\tau| \te{ and } \tau_{|_{|\omega|}}=\omega.\] For any two strings $\omega, \tau \in \sum^*,$ we say that $\omega$ and $\tau$ are incomparable if neither $\omega \prec \tau$ nor $\tau \prec \omega.$
	A subset $\Gamma \subset \sum^*$ is called a finite maximal antichain in $\sum^*$ if every sequence in $\sum$ is an extension of some string in $\Gamma$, but no element of $\Gamma$ is comparable to any other element in $\Gamma$. Building upon the work of \cite{PRV}, we fix some definitions. For $\omega \in \sum^*$, 
	\[
	\mathcal{E}_{\omega}(n):= \{\tau \in \sum\nolimits^{|\omega| + n}: \omega \prec \tau\} \quad \text{and} \quad \mathcal{E}_{\omega}^*:= \bigcup_{n \in \mathbb{N}} \mathcal{E}_{\omega}(n).
	\]
 	﻿Let $\Gamma$ be a finite maximal antichain. Then, we define $m_{\Gamma}$ and $M_{\Gamma}$ as the minimum and maximum lengths of the strings in $\Gamma$, respectively. For $\omega \in \sum^{m_{\Gamma}}$, we define
	\[
	\gD_{\gG}(\omega):= \{\tau \in \sum\nolimits^*: \omega \prec \tau, \mathcal{E}_{\tau}^* \cap \Gamma \neq \emptyset\}, \quad \gD_{\gG}^*:= \bigcup_{\omega \in \sum^{m_{\Gamma}}} \gD_{\gG}(\omega).
	\]
	For any $\omega = \omega_1 \omega_2\cdots \omega_n \in \sum^*$, we define $\omega_{-} = \omega_1 \omega_2\cdots\omega_{n-1}$ as the string obtained by removing the last symbol from $\omega$, and $\omega_{-} = \emptyset$ if $n = 1$. For $\omega=\go_1\go_2\cdots\go_n \in \sum^*$, we define $f_{\omega} = f_{\omega_1} \circ f_{\omega_2} \circ \cdots \circ f_{\omega_n}$ if $n\geq 1$, and identify $f_\go$ as the identity mapping $I_{\D R^d}$ on $\D R^d$ if $n=0$. 
	For $n \ge 1,$ by iterating on Equations \eqref{eq2} and \eqref{eq3}, we obtain
	$$K = \bigg(\bigcup_{\omega\in \sum^n} f_{\omega}(K) \bigg) \cup \bigg( \bigcup_{m=0}^{n-1} \bigcup_{\omega \in \sum^m} f_{\omega} (C) \bigg), ~~ \text{and}~~ \mu= \sum_{\omega \in \sum^n} p_\omega \mu \circ f_{\omega}^{-1} + p_0 \sum_{m=0}^{n-1} \sum_{\omega \in \sum^m} p_{\omega} \nu \circ f_{\omega}^{-1}.$$
	﻿	\begin{remark}
		In \cite{R6}, Roychowdhury has shown the existence of the quantization dimension of such a measure by establishing a relationship between the quantization dimension function and the temperature function.
	\end{remark} 
	Now, we recall some lemmas and propositions that will be needed to prove our main result.
	\begin{prop}(see \cite{GL2})\label{prop01}
		\begin{enumerate}
			\item If $0 \le t_1 < \overline{D}_r < t_2,$ then 
			\[ \limsup_{n \to \infty} n V_{n,r}^{\frac{t_1}{r}} = \infty ~ \text{and}~~ \lim_{n \to \infty} n V_{n,r}^{\frac{t_2}{r}}=0.\]
			\item If $0 \le t_1 < \underline{D}_r < t_2,$ then
			\[ \liminf_{n \to \infty} n V_{n,r}^{\frac{t_2}{r}} = 0 ~ \text{and}~~ \lim_{n \to \infty} n V_{n,r}^{\frac{t_1}{r}}=\infty.\]
		\end{enumerate}
	\end{prop}
	\begin{lemma}(see \cite[Lemma ~ 3.2]{R6})\label{lemma1000}
		Let $\nu$ be the image measure on the self-conformal set $C$ of the shift-invariant ergodic measure $\hat{\nu}$ on the coding space under the coding map. Then the quantization dimension of order $r$ of $\nu$ is $k_r$, where $k_r$ is uniquely determined by the relation 
		\begin{equation}\label{eq97}
			\lim_{n \to \infty} \frac{1}{n} \log \sum_{|\omega|=n} (\|g_\omega'\|^r \hat{\nu}[\omega])^{\frac{k_r}{r+k_r}}=0.
		\end{equation}
	\end{lemma}
	\begin{lemma}(see \cite[Lemma~ 3.3 \te{ and } Lemma~3.4]{R6}) \label{lemma2.4}
		Let $ 0 < r < +\infty $ and $ k_r $ be the number defined in Lemma \ref{lemma1000}. Then, for any $ n \ge 1 $, we have
		\begin{equation} \label{eq96}
		 \sum_{|w|=n} \left( \|g'_{\omega}\|^r \hat{\nu}[\omega] \right)^t \le e^{nh(t)} (c^r L)^t,
		\end{equation}
		where $ h(t) = \lim_{n \to \infty} \frac{1}{n} \log \sum_{|w|=n} \left( \|g'_{\omega}\|^r \hat{\nu}[\omega] \right)^t $. Moreover, for $ t = \frac{k_r}{r + k_r} $,
		\[
		\sum_{|w|=n} \left( \|g'_{\omega}\|^r \hat{\nu}[\omega] \right)^{k_r / (r + k_r)} \le (c^{r} L)^{k_r / (r + k_r)}.
		\]
	\end{lemma}
	\begin{lemma}(see \cite{Z3})
		Let $\Gamma$ be a finite maximal antichain. Then,
		\[ K = \Big(\bigcup_{\omega \in \Gamma} f_{\omega}(K)\bigg) \cup \bigg( \bigcup_{\omega \in \gD_{\Gamma}^*} f_{\omega}(C)\bigg) \cup \bigg( \bigcup_{n=0}^{m_\Gamma -1} \bigcup_{\omega \in \sum^n} f_{\omega}(C)\Big).\]
	\end{lemma}
Let us recall the definition of Voronoi region. Which will be used in the proofs of the results in Section \ref{sec4}.
\begin{defi}
	Let $(X,d)$ be a metric space and $A \subseteq X$. Then, the Voronoi region generated by a point $a \in A$ is defined as 
	$M(a| A)= \{x \in X : d(x,a)= \min_{b \in A} d(x,b)\}.$
	The collection $\{M(a | A): a \in A\}$  is called the Voronoi diagram of $A$.
\end{defi}
	\section{main result} \label{sec3}
	In this section, we state and prove the following theorem, which is the main result of the paper. 
	\begin{theorem}\label{thm1}
		Let $r$ be a positive real number, and let $\mu$ be the condensation measure associated with the condensation system $(\{f_i\}_{i=1}^N, (p_i)_{i=0}^N, \nu)$, where $C$ is the conformal set generated by a conformal IFS $\mathcal{I}_1 = (\mathbb{R}^d; \{g_j\}_{j=1}^M)$, and $\gn$ is the image measure of an ergodic measure with bounded distortion on the code space $\set{1, 2, \cdots, M}^{\D N}$ with support $C$. Assume that the condensation system satisfies the SSC. Then,
		\[
		\max\{d_r, D_r(\nu)\} \le \underline{D}_r(\mu),
		\]
		where $d_r$ is uniquely determined by the relation	$\sum_{i=1}^N \left(p_i a_i^r\right)^{\frac{d_r}{r + d_r}}=1.$ Furthermore, if the IFS $\mathcal{I}_1$ satisfies the SSC, then
		\[
		\max\{d_r, D_r(\nu)\} \le \underline{D}_r(\mu) \le \overline{D}_r(\mu) \le \max\{l_r, D_r(\nu)\},
		\]
		where $l_r$ is determined by the relation
		$\sum_{i=1}^N \left(p_i b_i^r\right)^{\frac{l_r}{r + l_r}}=1.$	
	\end{theorem}

	\begin{remark}\label{rem30} By \cite[Theorem~1.1]{PRV}, we know that for $r\in (0,  +\infty)$, if $\mu$ is the condensation measure associated with the condensation system
		$(\{f_i\}_{i=1}^{N}, (p_i)_{i=0}^{N},\nu)$, where $\gn$ is any Borel probability measure with compact support $C$, then 
		\begin{equation} \label{Sh1} \max\{k_r, \underline{D}_r(\nu)\}\leq \underline{D}_r(\mu).
		\end{equation} 
		In this paper, the Borel probability measure $\gn$ is the image measure of an ergodic measure with bounded distortion on the code space $\set{1, 2, \cdots, M}^{\D N}$ with support a conformal set $C$. Also, we know that for such an image measure $\gn$, the quantization dimension $D(\gn)$ exists (see \cite[Theorem~3.1]{R6}), and hence the relation~\eqref{Sh1}, in this paper, reduces to 
		\[\max\{d_r, D_r(\nu)\} \le \underline{D}_r(\mu).\]
		Thus, in the sequel of this section, to complete the proof of Theorem~\ref{thm1}, we only prove the following relation: 
		\[\overline{D}_r(\mu) \le \max\{l_r, D_r(\nu)\}.\]
		\qed
	\end{remark}
	Motivated by \cite{PRV}, we define the finite maximal antichains in $ \sum^* $ and $ \sum_1^* $, respectively. ﻿Let $c_j = \|g_j'\|^r \hat{\nu}[j]$ and $\xi_r= \min\{\min_{1 \le j \le M} \|g_j'\|^r \hat{\nu}[j], \min_{1 \le j \le N} p_j b_j^r \}.$ Then, for $n \ge 1$ we define a finite maximal antichain in $\sum^*.$ Write 
	\[\Gamma_n^r:= \{\omega \in  \sum\nolimits_{1}^{*} : p_\omega b_\omega^r < \xi_r^n \le p_{\omega_{-}} b_{\omega_{-}}^r\}.\]
	Let $|\Gamma_n^r|$ denote the cardinality of the set $\Gamma_n^r$ then for all $n\ge 1,$ we write
	\begin{equation} \label{eq98} 
		\Phi_n^r:= \bigcup_{m=0}^{m_{\Gamma_n^r}-1} \sum\nolimits_1^m \cup \gD_{\Gamma_n^r}^*.
	\end{equation}
	It is an easy observation that for any $\omega \in  \Phi_n^r,$ we have $p_\omega b_{\omega}^r \ge \xi_r^n.$ Thus, for any $\omega \in \Phi_n^r,$ we define a finite maximal antichain in $\sum_1^*$ as follows: 
	$$\Gamma_n^r(\omega):= \{\tau \in \sum\nolimits_1^*: p_{\omega} b_{\omega}^r c_{\gt} < \xi_r^n \le p_{\omega} b_{\omega}^r c_{\gt_{-}} \}.$$
	Let $|\Gamma_n^r(\omega)|$ denote the cardinality of $\Gamma_n^r(\omega)$ for each $\omega \in \Phi_n^r.$ We define
	\begin{equation} \label{eq99}
		\psi_{n,r} := |\Gamma_n^r|+ \sum_{\omega \in \Phi_n^r} |\Gamma_n^r(\omega)|.
	\end{equation}
	﻿\begin{lemma}(see \cite{PRV})
		For each $n \ge 1,$ we have 
		$$K = \bigg( \bigcup_{\omega \in \Gamma_n^r} f_{\omega}(K)\bigg) \cup \bigg( \bigcup_{\omega \in \gD_{\Gamma_n^r}^*} f_{\omega}(C)\bigg) \cup \bigg( \bigcup_{m=0}^{m_{\Gamma_n^r}-1} \bigcup_{\omega \in \Sigma^m} f_{\omega}(C)\bigg).$$
	\end{lemma}
	\begin{lemma}\label{lemma99}
		Assume that the condensation system $(\{f_i\}_{i=1}^N, (p_i)_{i=0}^N, \nu)$ and IFS $\mathcal{I}_1 = (\mathbb{R}^d; \{g_j\}_{j=1}^M)$ satisfy the SSC. Then, we have $$V_{\psi_{n,r}r}(\mu) \le d^*\bigg(\sum_{\omega \in \Gamma_{n}^r} p_\omega b_{\omega}^r+ \sum_{\omega \in \Phi_{n}^r}\sum_{\tau \in \Gamma_{n}^r(\omega)} p_{\omega} b_{\omega}^r c \|g'_{\tau}\|^r \hat{\nu}([\tau]) \bigg),$$
		where $d^*= \max\{(diam(K))^r,(diam(C))^r\},$ and $\psi_{n,r}$ is given by Equation \eqref{eq99}.
	\end{lemma}
	\begin{proof}
		We construct a set $A$ by choosing an element $a_\omega \in f_{\omega}(K)$ for each $\omega \in \Gamma_{n}^r$. In addition, for each $\omega \in \Phi_{n}^r,$ we choose an element $a_{\omega,\tau} \in f_{\omega}(g_{\tau}(C))$ such that $\tau \in \Gamma_{n}^r(\omega)$. Clearly, $\text{card}(A) \le \psi_{n,r}.$ We have 	\begin{align*}
			V_{\psi_{n,r}r}(\mu) &\le \int_{K} d(x,A)^r d\mu(x) \\
			&\le \sum_{\omega \in \Gamma_{n}^r} \int_{f_{\omega}(K)} d(x,A)^r d\mu(x) + \sum_{\omega \in \Phi_{n}^r} \sum_{\sigma \in \Gamma_{n}^r(\omega)} \int_{f_{\omega}(g_{\tau}(C))} d(x,A)^r d\mu(x) \\
			&\le \sum_{\omega \in \Gamma_{n}^r} \int_{f_{\omega}(K)} d(x,a_{\omega})^r d\mu(x) + \sum_{\omega \in \Phi_{n}^r} \sum_{\tau \in \Gamma_{n}^r(\omega)} \int_{f_{\omega}(g_{\tau}(C))} d(x,a_{\omega,\tau})^r d\mu(x) \\
			&\le \sum_{\omega \in \Gamma_{n}^r} (\text{diam}(f_{\omega}(K)))^r \mu(f_{\omega}(K)) + \sum_{\omega \in \Phi_{n}^r} \sum_{\tau \in \Gamma_{n}^r(\omega)} (\text{diam}(f_{\omega}(g_{\tau}(C))))^r \mu(f_{\omega}(g_{\tau}(C))) \\
			&\le \sum_{\omega \in \Gamma_{n}^r} b_{\omega}^r (\text{diam}(K))^r \mu(f_{\omega}(K)) + \sum_{\omega \in \Phi_{n}^r} \sum_{\tau \in \Gamma_{n}^r(\omega)} b_{\omega}^r c\|g'{_\tau}\|^r (\text{diam}(C))^r \mu(f_{\omega}(g_{\tau}(C))) \\& \le d^* \sum_{\omega \in \Gamma_{n}^r} b_{\omega}^r  \mu(f_{\omega}(K)) + d^* \sum_{\omega \in \Phi_{n}^r} \sum_{\tau \in \Gamma_{n}^r(\omega)} b_{\omega}^r c\|g'{_\tau}\|^r \mu(f_{\omega}(g_{\tau}(C))),
		\end{align*}
		where $d^* = \max\{(\text{diam}(K))^r, (\text{diam}(C))^r\}.$ Now, we find out the values of $\mu(f_{\omega}(K))$ and $\mu(f_{\omega}(g_{\tau})(C))$ for $\omega \in \Gamma_{n}^r$ and $\tau \in \Gamma_{n}^r(\omega).$
		Let $\omega \in \Gamma_{n}^r$ and $|\omega|=n_1.$ Since $\mu= \sum_{\sigma\in \sum^{n_1}} p_{\sigma} \mu \circ f_{\sigma}^{-1} + p_0 \sum_{m=0}^{n_1-1} \sum_{\sigma \in \sum^m} p_{\sigma} \nu \circ f_{\sigma}^{-1},$ and the condensation system satisfy the SSC, we have 
		\begin{align*}
			\mu(f_{\omega}(K)) &= \sum_{\sigma \in \sum^{n_1}} p_{\sigma} \mu \circ f_{\sigma}^{-1}(f_{\omega}(K)) + p_0 \sum_{m=0}^{n_1 -1} \sum_{\sigma\in \sum^m} p_{\sigma} \nu \circ f_{\sigma}^{-1}(f_{\omega}(K)) \\& = p_{\omega} \mu(f_{\omega}^{-1}(f_{\omega}(K))) + p_0 \sum_{m=0}^{n_1 -1} \sum_{\sigma \in \sum^m} p_{\sigma} \nu \circ f_{\sigma}^{-1}(f_{\omega}(K)).
		\end{align*}
		Since the condensation system satisfies the SSC, we have $\mu(f_{\omega}(K))= p_{\omega}.$
		Also,
		\begin{align*}
			\mu(f_{\omega}(g_{\tau}(C))) &= \sum_{\sigma \in \sum^{n_1}} p_{\sigma} \mu \circ f_{\sigma}^{-1}(f_{\omega}(g_{\tau}(C))) + p_0 \sum_{m=0}^{n_1-1} \sum_{\sigma \in \sum^m} p_{\sigma} \nu \circ f_{\sigma}^{-1}(f_{\omega}(g_{\tau}(C))) \\&= p_{\omega}(\mu(g_{\tau}(C))) +0 \\&= p_{\omega}\bigg(p_0 \nu(g_{\tau}(C))+ \sum_{i=1}^N p_i \mu \circ f_{i}^{-1}(g_{\tau}(C))\bigg).
		\end{align*}
		By using Lemma \ref{lemma2.4} and the fact that IFS $\mathcal{I}_1$ satisfies SSC, we have 
		\begin{align*}
			\mu(f_{\omega}(g_{\tau}(C))) &= p_{\omega} p_0 \nu(g_{\tau}(C)) \\& \le p_{\omega} p_0 c \sum_{\tau' \in \Gamma_n^{r}(\omega)} \hat{\nu}([\tau']) \nu \circ g_{\tau'}^{-1}(g_{\tau}(C)) \\& = p_{\omega} p_0 c \hat{\nu}([\tau])\nu(C) = p_{\omega} p_0 c \hat{\nu}([\tau]) \leq  p_{\omega} c \hat{\nu}([\tau]).
		\end{align*}
		By using the above inequalities, we get the required result.

	\end{proof}
	\begin{prop}\label{prop00}
		Let the condensation system $(\{f_i\}_{i=1}^N,(p_i)_{i=0}^N, \nu),$ and IFS $\mathcal{I}_1:=(\mathbb{R}^d;\{g_j\}_{j=1}^M)$ satisfy the SSC. Let $\mu$ be the condensation measure associated with the condensation system $(\{f_i\}_{i=1}^N,(p_i)_{i=0}^N, \nu)$. Then, for all $s > \max\{l_r,k_r\}$, we have 
		\[
		\limsup_{n \to \infty}  nV_{n,r}^{\frac{s}{r}} (\mu) < \infty,\] where $l_r$ is determined by the relation $\sum_{i=1}^N (p_i b_i)^{\frac{l_r}{r+l_r}}=1$ and $k_r$ is uniquely determined by the expression \eqref{eq97}.  
	\end{prop}
	\begin{proof}
		By {Lemma \ref{lemma99}} and using the definition of $\Gamma_n^r$ and $\Gamma_n^r(\omega),$ we have,
		\begin{equation}\label{eq19}
			V_{\psi_{n,r},r}(\mu) \le d^*\bigg(\sum_{\omega \in \Gamma_{n}^r} p_\omega b_{\omega}^r+ \sum_{\omega \in \Phi_{n}^r}\sum_{\tau \in \Gamma_{n}^r(\omega)} p_{\omega} b_{\omega}^r c \|g'_{\tau}\|^r \hat{\nu}([\tau]) \bigg) \le d^* c \xi_r^n \psi_{n,r}.
		\end{equation}
		Now, we aim to estimate the bound for $\psi_n^r.$ Since $\sum_{i=1}^N (p_i b_i^r)^{\frac{l_r}{r + l_r}}=1$ and $\Gamma_n^r$ is finite maximal antichain in $\sum^*,$ we have $\sum_{\go \in \Gamma_n^r}(p_{\go} b_{\go})^{\frac{l_r}{r + l_r}}=1.$ Let $s \ge \{l_r, k_r\}.$ Then, we have 
		\[
		\sum_{\go \in \Gamma_n^r}(p_{\go} b_{\go})^{\frac{s}{r + s}}<1.
		\]
		By using the definitions of $\psi_{n}^r, \Gamma_{n}^r$, and $\Gamma_{n}^r(\go)$ for each $\go \in \Phi_{n}^r,$ we have
		\begin{align*}
			\psi_{n,r} \xi_r^{\frac{(n+1)s}{r+s}} &\le \sum_{\omega \in \Gamma_n^r} (p_{\omega}b_{\omega}^r)^{\frac{s}{r+s}} + \sum_{\omega \in \Phi_{n}^r} \sum_{\tau \in \Gamma_{n}^{r}(\omega)} (p_{\omega} b_{\omega}^r \|g_{\tau}'\|^r \hat{\nu}([\tau]))^{\frac{s}{r+s}} \\& \le \sum_{\omega \in \Gamma_n^r}(p_{\omega}b_{\omega}^r)^{\frac{s}{r+s}} + \sum_{\omega \in \Phi_{n}^r} (p_{\omega} b_{\omega}^r)^{\frac{s}{r+s}}\sum_{\tau \in \Gamma_{n}^{r}(\omega)} ( \|g_{\tau}'\|^r \hat{\nu}([\tau]))^{\frac{s}{r+s}} \\&\le 1+ e^{M_{\Gamma_n^r(\go)}h(s/r+s)} (c^{r} L)^{\frac{s}{r+s}} \sum_{\omega \in \Phi_n^r} (p_{\omega} b_{\omega}^r)^{\frac{s}{r+s}} \hspace{3cm} \text{(by Equation~\eqref{eq96})}\\& \le 1+ e^{M_{\Gamma_n^r(\go)}h(s/r+s)} (c^{r} L)^{\frac{s}{r+s}} \bigg(\sum_{m=0}^{m_{\Gamma_{n}^r}-1}\sum_{\omega \in \sum^m} (p_{\omega} b_{\omega}^r)^{\frac{s}{r+s}}+ \sum_{m=m_{\Gamma_{n}^r}}^{M_{\Gamma_{n}^r}}\sum_{\omega \in \sum^m} (p_{\omega} b_{\omega}^r)^{\frac{s}{r+s}} \bigg) \\& = 1+ e^{M_{\Gamma_n^r(\go)}h(s/r+s)} (c^{r} L)^{\frac{s}{r+s}} \bigg( \sum_{m=0}^{M_{\Gamma_{n}^r}}\sum_{\omega \in \sum^m} (p_{\omega} b_{\omega}^r)^{\frac{s}{r+s}} \bigg) \\& \le e^{M_{\Gamma_n^r(\go)}h(s/r+s)} (c^{r} L)^{\frac{s}{r+s}}\bigg( 1+ \sum_{m=0}^{M_{\Gamma_n^r}}\bigg(\sum_{i=1}^N (p_i b_i^r)^{\frac{s}{r+s}}\bigg)^m\bigg) \\& \le e^{M_{\Gamma_n^r(\go)}h(s/r+s)} (c^{r} L)^{\frac{s}{r+s}} \bigg(\frac{2}{1-\sum_{i=1}^{N}(p_i b_i^r)^{\frac{s}{r+s}}}\bigg).
		\end{align*}
		Combining the above inequality with Equation~\eqref{eq19}, we have
		\begin{align*}
			\psi_{n,r} V_{\psi_{{n,r},r}}^{\frac{s}{r}} (\mu) &\le (d^*)^{\frac{s}{r}} \xi_{r}^{\frac{ns}{r}} \psi_{n,r}^{\frac{s+r}{r}} \\& \le (d^*)^{\frac{s}{r}} \xi_{r}^{\frac{-s}{r}}e^{M_{\Gamma_n^r(\go)}h(s/r+s)} (c^{r} L)^{\frac{s}{r+s}} \bigg(\frac{2}{1-\sum_{i=1}^{N}(p_i b_i^r)^{\frac{s}{r+s}}}\bigg)^{\frac{s+r}{r}}
			\\& \le (d^*)^{\frac{s}{r}} \xi_{r}^{\frac{-s}{r}} (c^{r} L)^{\frac{s}{r+s}} \bigg(\frac{2}{1-\sum_{i=1}^{N}(p_i b_i^r)^{\frac{s}{r+s}}}\bigg)^{\frac{s+r}{r}}.
		\end{align*}
		The last arrived as $h(t)$ is a decreasing function in $t$ and $s> k_r.$ Hence, for all $s > \max\{l_r,k_r\}$, we can deduce that  $\limsup_{n \to \infty} \psi_{n,r} V_{\psi_{n,r}}^{\frac{s}{r}} < \infty.$ By using \cite[Lemma~2.4]{Z4}, we can infer that 
		$$\limsup_{n \to \infty} n V_{n,r}^{\frac{s}{r}} < \infty  \te{ for all } s > \max\{l_r,k_r\}.$$ Thus, the proof of the proposition is complete.
	\end{proof} 
	\subsection*{Proof of the main theorem Theorem~\ref{thm1}}
	From Proposition~\ref{prop01} and Proposition~\ref{prop00}, we have $\overline{D}_r(\mu) \leq \max\{l_r, k_r\}$. Also, from Remark~\ref{rem30}, it follows that $\max\{d_r, D_r(\nu)\} \leq \underline{D}_r(\mu)$. Further, by \cite[Theorem~3.1]{R6}, the quantization dimension of $\nu$ exists and is equal to $k_r$, i.e., $D_r(\nu) = k_r$. Hence,  
	\[
	\max\{d_r, D_r(\nu)\} \leq \underline{D}_r(\mu) \leq \overline{D}_r(\mu) \leq \max\{l_r, D_r(\nu)\}.
	\]  
	This completes the proof of the theorem.
	\qed 
		\begin{example}
		Let $ \mathcal{I}_1 = \{ f_1, f_2 \} $ be an IFS, where $ f_1(x) = \frac{x}{3} $ and $ f_2(x) = \frac{x+2}{3} $.  The invariant set corresponding to $ \mathcal{I}_1 $ is a Cantor set contained in $ [0, 1] $. Let $ \mathcal{I}_2 = \{ g_1, g_2 \} $ be a conformal IFS, where
		\[
		g_1(x) = \frac{2}{5} + \frac{1}{3} \left( x - \frac{2}{5} \right), \te{ and } g_2(x) = \frac{3}{5} + \frac{1}{3} \left( x - \frac{3}{5} \right).
		\]
		Let $ \nu $ be an image measure of a shift-invariant ergodic measure on the symbolic space $\set{1, 2}^{\D N}$. Since 
		\[
		g_i \left( \left[ \frac{2}{5}, \frac{3}{5} \right] \right) \subset \left[ \frac{2}{5}, \frac{3}{5} \right],
		\]
		the support of $ \nu $ is contained in $ \left[ \frac{2}{5}, \frac{3}{5} \right] $. Now, we can consider the corresponding condensation system as $(\{f_i\}_{i=1}^2, (\frac 13, \frac 13, \frac 13), \nu)$. It is easy to observe that the inhomogeneous invariant set is contained in $ [0, 1] $. To check the validity of Theorem~\ref{thm1}, it is sufficient to check the SSC for the condensation system $(\{f_i\}_{i=1}^2, (\frac 13, \frac 13, \frac 13), \nu)$, and the SSC for the IFS $ \mathcal{I}_2 $. We have,
		\[
		f_1([0, 1]) \cap f_2([0, 1]) = \emptyset,
		\te{ and }  
		f_i([0, 1]) \cap \left[ \frac{2}{5}, \frac{3}{5} \right] = \emptyset.
		\]
		Also, 
		\[
		g_1\left( \left[\frac{2}{5}, \frac{3}{5}\right] \right) \cap g_2\left( \left[\frac{2}{5}, \frac{3}{5}\right] \right) = \emptyset.
		\]
		Thus, the conditions in Theorem \ref{thm1} hold.
		Then, $ l_2 = d_2 $ is given by
		\[
		2 \left( \frac{1}{3} \cdot \frac{1}{3}^2 \right)^{\frac{l_2}{2 + l_2}} = 1.
		\]
		Solving the above for $ l_2 $, we get
		\[
		l_2 = \frac{2 \log \frac{1}{2}}{\log \frac{2}{27}} \approx 0.5326.
		\]
	\end{example}
	\section{Optimal quantization for an infinite discrete distribution} \label{sec4}
	In this section, we investigate the optimal quantization for an infinite discrete distribution with respect to the squared Euclidean distance, i.e., when $r=2$. 
	Let $\gn$ be an infinite discrete distribution with support $C:=\set{a_j : j\in \D N}$, where $a_j:=\frac{2}{5}+\frac 1 5\frac{2^{j-1}-1}{2^{j-1}}$. Notice that $C=\left\{\frac{2}{5},\frac{1}{2},\frac{11}{20},\frac{23}{40},\frac{47}{80},\frac{19}{32},\frac{191}{320},\frac{383}{640}, \cdots \right\}\sci [\frac 25, \frac 35]\sci \D R$. Let the probability mass function of $\gn$ be given by $f$ such that
	\[f(x)=\left\{\begin{array}{cc}
		\frac 1{2^j} & \te{ if }   x=a_j \te{ for } j\in \D N,  \\
		0 &  \te{ otherwise}.
	\end{array}\right.\]
	In this section, we determine the optimal sets of $n$-means and the $n$th quantization errors for all $n\in \D N$, and the quantization dimension of $\gn$ defined as above. 
	Set
	$A_k:=\set{a_k, a_{k+1}, a_{k+2}, \cdots}$, and $b_k=E(X: X \in A_k)$, where $k\in \D N$. Then, by the definition of conditional expectation, we see that
	\[b_k=\frac{\sum_{j=k}^\infty a_j\frac 1{2^j}}{\sum_{j=k}^\infty \frac 1{2^j}}.\]
	For any Borel measurable function $g$ on $\D R$, by $\int g d\gn$, it is meant that
	\[\int g d\gn=\sum_{j=1}^\infty g(a_j) \gn(a_j)=\sum_{j=1}^\infty g(a_j)\frac 1{2^j}.\]
	﻿\begin{lemma}\label{lemma001}
		For a probability distribution $\gn$, let $E(\gn)$ represent the expected value of a random variable with distribution $\gn$, and $V(\gn)$ represent its variance. Then,
		$E(\gn)=\frac{7}{15}, \te{ and } V(\gn)=\frac{8}{1575}.$
	\end{lemma}
	﻿
	\begin{proof} Using the definitions of expected value and the variance of a random variable, we have
		\[E(\gn)=\sum_{j=1}^\infty a_j \frac 1{2^j}=\frac{7}{15}, \te{ and } V(\gn)=\sum_{j=1}^\infty (a_j-E(\gn))^2 \frac 1{2^j}=\frac{8}{1575},\]
		and thus the proof of the lemma is complete.
	\end{proof}
	﻿
	\begin{note}
		By Lemma~\ref{lemma001}, it follows that the optimal set of one-mean for the discrete distribution $\gn$ consists of the expected value $\frac{7}{15}$ and the corresponding quantization error is the variance $V(\gn)$ of $\gn$, i.e., $V_1(\gn)=V(\gn)=\frac{8}{1575}$.
	\end{note}
	The following proposition is known. 
	\begin{prop}  (see \cite{GG, GL1})\label{prop1000}
		Let $\ga$ be an optimal set of $n$-means, $a \in \ga$, and $M(a|\ga)$ be the Voronoi region generated by $a\in \ga$.
		Then, for every $a \in\ga$,
		$(i)$ $\mu(M(a|\ga))>0$, $(ii)$ $ \mu(\partial M(a|\ga))=0$, $(iii)$ $a=E(X : X \in M(a|\ga))$, and $(iv)$ $\mu$-almost surely the set $\set{M(a|\ga) : a \in \ga}$ forms a Voronoi partition of $\D R^d$.
	\end{prop}
	\begin{lemma} \label{lemma0002} The set $\set{\frac 25, \frac {8}{15}}$ forms an optimal set of two means for $\gn$, and the corresponding quantization error is given by $V_2(\gn)=\frac{1}{1575}=0.000634921$.
	\end{lemma}
	﻿
	\begin{proof} Let $\gb:=\set{\frac 25, \frac {8}{15}}$. The distortion error due to the set $\gb$ is given by
		\[\sum_{j=1}^\infty \min_{a\in \gb}(a_j-a)^2 \frac 1{2^j}=\frac{1}{1575}=0.000634921.\]
		Since $V_2(\gn)$ is the quantization for two-means, we have $V_2(\gn)\leq 0.000634921$. Let $\ga:=\set{c_1, c_2}$ be an optimal set of two-means for $\gn$ such that $c_1<c_2$. Since the elements in an optimal set are the expected values in their own Voronoi regions, we have $\frac 25\leq c_1<c_2\leq \frac 35$. Suppose that $\frac 25+\frac 1{20}< c_1$. Then,
		\[V_2(\gn)\geq (a_1-c_1)^2f(a_1)=(\frac 25-c_1)^2 \frac 12\geq \frac 1{800}=0.00125>V_2(\gn),\]
		which is a contradiction, and so $c_1\leq (\frac 25+\frac 1{20})$. We now show that the Voronoi region of $c_1$ does not contain any point from $A_2$. For the sake of contradiction, assume that the Voronoi region of $c_1$ contains a point, say $a_2$, from $A_2$. Then, we have
		\[V_2(\gn)\geq \min\left\{(a_1-c_1)^2\frac 1 2 +(a_2-c_1)^2 \frac 1{2^2} : \frac 25\leq c_1\leq (\frac 25+\frac 1{20})\right\}=\frac{1}{600}=0.00166667,\]
		which occurs when $c_1=\frac{13}{30}$. But, then $V_2(\gn)\geq 0.00166667>V_2(\gn)$, which leads to a contradiction. Hence, the Voronoi region of $c_1$ does not contain any point from $A_2$. Again, by Proposition~\ref{prop1000}, $P(M(c_1|\ga))>0$, and so the Voronoi region of $c_2$ does not contain the point $a_1$. Hence, $c_1=\frac 25$ and $c_2=E(X : X\in A_2)=\frac {8}{15}$, and the corresponding quantization error is given by $V_2(\gn)=\frac{1}{1575}=0.000634921$. Thus, the lemma is yielded. \end{proof}
	The following proposition, for all $n\geq 2$, gives the optimal sets of $n$-means and the corresponding quantization errors for $\gn$.
	﻿
	\begin{prop} \label{prop0003}
		For $j\in \D N$, let $a_j$, $A_j$, and $b_j$ be defined as before. Let $n\geq 2$. Then, $\ga_n:=\ga_n(\gn)=\set{a_1, a_2, \cdots, a_{n-1}, b_n}$ is an optimal set of $n$-means for $\gn$ and the corresponding quantization error is given by $V_n(\gn)=\frac{2^{6-3n}}{1575}$.
	\end{prop}
	﻿
	\begin{proof} Let us prove the proposition by induction.
		By Lemma~\ref{lemma0002}, the proposition is true if $n=2$. Let the proposition be true for $n=k$ for some positive integer $k\geq 2$, i.e., $\set{a_1, a_2, \cdots, a_{k-1}, b_k}$ is an optimal set of $k$-means for $\gn$ and the corresponding quantization error is given by $V_k(\gn)=\frac{2^{6-3k}}{1575}$. Let $\set{c_1, c_2, \cdots, c_k, c_{k+1}}$ be an optimal set of $(k+1)$-means for $\gn$ such that $\frac 25\leq c_1<c_2<\cdots<c_k<c_{k+1}\leq \frac 35$.
		Consider the set of $(k+1)$ points $\gb:=\set{a_1, a_2, \cdots, a_{k}, b_{k+1}}$. The distortion error due to the set $\gb$ is given by
		\[\sum_{j=k+1}^\infty (a_j-b_{k+1})^2 f(a_j)=\frac{2^{3-3 k}}{1575}.\]
		Since $V_{k+1}(\gn)$ is the quantization error for $(k+1)$-means, we have $V_{k+1}(\gn)\leq \frac{2^{3-3 k}}{1575}$.
		We show that $c_j=a_j$ for $1\leq j\leq k$. If $c_1\neq a_1$, then the Voronoi region of $c_1$ will contain at least $a_1$ and $a_2$, and so, \begin{align*} V_{k+1}(\gn)&\geq  \min\left\{(a_1-c_1)^2\frac 1 2 +(a_2-c_1)^2 \frac 1{2^2} : \frac 25\leq c_1\leq \frac 3 5\right\}=\frac{1}{600}=0.00166667,
		\end{align*}
		implying $V_{k+1}(\gn)\geq 0.00166667>V_2(\gn)\geq V_k(\gn)>V_{k+1}(\gn)$, which is a contradiction. So, we can assume that $c_1=a_1$. Suppose that $\ell$, where $2\leq \ell\leq k$, is the least positive integer for which $c_\ell\neq a_\ell$. Then, $c_1=a_1, c_2=a_2, \cdots, c_{\ell-1}=a_{\ell-1}$, and the Voronoi region of $c_\ell$ must contain at least $a_\ell$ and $a_{\ell+1}$ yielding
		\begin{align*} V_{k+1}(\gn)&\geq  \min\left\{(a_\ell-c_\ell)^2 \frac 1{2^{\ell}}+(a_{\ell+1}-c_\ell)^2\frac 1{2^{\ell+1}} : a_\ell \leq c_{\ell}\leq \frac 35\right\}=\frac{8^{-\ell}}{75},
		\end{align*}
		which occurs when $c_\ell=\frac{3}{5}-\frac{2^{-\ell}}{3}$.
		Since $2\leq \ell\leq k$, we have
		\[\frac {2^{3-3k}}{1575}\leq \frac {2^{3-3\ell}}{1575}\leq \frac 8{21} \frac{8^{-\ell}}{75}<\frac{8^{-\ell}}{75}\]  implying
		$V_{k+1}(\gn)\geq \frac{8^{-\ell}}{75}>\frac {2^{3-3k}}{1575}\geq V_{k+1}(\gn)$, which yields a contradiction. Hence, we can assume that $c_1=a_1, c_2=a_2, \cdots, c_k=a_k$ yielding $c_{k+1}=b_{k+1}$ whenever it is true for $n=k$. Thus, the proposition is true for $n=k+1$. Hence, by the Principle of Mathematical Induction, the proposition is true for all positive integers $n\geq 2$.
	\end{proof}
	﻿
	﻿
	Let us now give the following theorem.
	
	﻿
	\begin{theorem} \label{theo2}  Let $\gn$ be the infinite discrete distribution. Then, the quantization dimension $D(\gn)$ of the measure $\gn$ exists and equals zero, but the quantization coefficient does not exist. 
	\end{theorem}
	﻿
	\begin{proof}
		By Proposition~\ref{prop0003}, we have $V_n(\gn)=\frac{2^{6-3n}}{1575}$, and so,
		$D(\gn)=\mathop{\lim}\limits_{n\to \infty} \frac {2\log n}{-\log V_n(\gn)}=0$, i.e., $D(\gn)$ exists and equals zero. Since $D(\gn)=0$, the quantization coefficient for $\gn$ does not exist. Thus, the proof of the theorem is complete. 
	\end{proof}
	\begin{note}
		A detailed study of the quantization dimension of a condensation measure corresponding to a condensation system (inhomogeneous IFS) $(\{f_i\}_{i=1}^N,(p_i){i=1}^N,\gn)$, where $\gn$ is a discrete distribution defined as above, can be found in \cite{DRV1}. 
	\end{note}
	Let us now give the following example. 	﻿
	\begin{example}{\label{exam23}}
		Let $\nu$ be an infinite discrete distribution with support $C := \{a_j : j \in \mathbb{N}\}$, where $a_j := 4 \cdot 3^{j-1}$. The probability mass function $f$ of $\nu$ is given by:
		\[
		f(x) = \begin{cases}
			\frac{k \left( 2(j-1) \log(j-1) + (j-2) \right)}{3^j (j-1)(j-2)^3 \log^2(j-1)} & \text{if } x = a_j \text{ for } j \geq 3, \\
			0 & \text{otherwise},
		\end{cases}
		\]
		where $k = \left( \sum_{j=2}^{\infty} \frac{(2(j-1) \log (j-1) + (j-2))}{3^{j} (j-1)(j-2)^3 \log^2 (j-1)} \right)^{-1} < \infty$ is the normalizing constant. For each $a_j$, we have $0 \leq f(a_j) \leq 1$ and $\sum_{j=3}^{\infty} f(a_j) = 1$. These conditions ensure that $f$ is a valid probability mass function.
		Let $A \subseteq \mathbb{R}$ with $\te{card}(A) = n$, and define the index set
		\[
		I = \{ j \geq 2 : A \cap [3^j, 3^{j+1}) = \emptyset \}.
		\]
		Then, for $j \in I$, we have
		\[
		\begin{aligned}
			\min_{a \in A} |a_j - a|^2 & \geq (a_j - 3^j)^2  = (4 \cdot 3^{j-1} - 3^j)^2  = 3^{2(j-1)}.
		\end{aligned}
		\]
		We see that 
		\[
		\begin{aligned}
			&\sum_{j=2}^{\infty} \min_{a \in A} |a_j - a|^2 \cdot \frac{k (2(j-1) \log(j-1) + (j-2))}{3^{2j} (j-1)(j-2)^3 \log^2(j-1)} \\&\geq \sum_{j \in I} 3^{2(j-1)} \cdot \frac{k (2(j-1) \log(j-1) + (j-2))}{3^{2j} (j-1)(j-2)^3 \log^2(j-1)}  \\
			&=  \frac{k}{9} \sum_{j=n+2}^{\infty} \frac{(2(j-1) \log(j-1) + (j-2))}{(j-1)(j-2)^3 \log^2(j-1)}\\
			& \geq \frac{k}{9} \int_{j=n+2}^{\infty} \frac{(2(x-1) \log(x-1) + (x-2))}{(x-1)(x-2)^3 \log^2(x-1)} \\
			& \geq \frac{k}{9 n^2 \log(n+2)}.
		\end{aligned}
		\]
		The above inequality is true for any $A\ci \D R$ with $\te{card}(A)=n$, and hence $V_{n, 2}(\gn)\geq \frac{k}{3^2 n^2 \log(n+2)}$. 
		Furthermore, for $B = \{a_2, \ldots, a_{n+1}\}$, we observe:
		\[
		\min_{b \in B} |a_j - b|^2 = 0 \te{ if }  j \leq n+1,
		\te{ and }
		\min_{b \in B} |a_j - b|^2 \leq a_j^2 = \frac{16 \cdot 3^{2j}}{9} \te{ if } j \geq n+2.
		\]
		Consequently,
		\[
		\begin{aligned}
			\sum_{j=2}^{\infty} \min_{b \in B} |a_j - b|^2 \cdot \frac{k (2(j-1) \log(j-1) + (j-2))}{3^{2j} (j-1)(j-2)^3 \log^2(j-1)} &\leq \frac{16k}{9} \sum_{j=n+2}^{\infty} \frac{(2(j-1) \log(j-1) + (j-2))}{(j-1)(j-2)^3 \log^2(j-1)} \\
			& \leq \frac{16k}{9 n^2 \log(n+1)},
		\end{aligned}
		\]
		implying $V_{n, 2}(\gn)\leq \frac{16k}{9 n^2 \log(n+1)}$. Therefore, we conclude:
		\[
		\begin{aligned}
			&-\log \frac{16 k}{9 n^2 \log(n+1)}  \leq - \log V_{n,2}(\nu) \leq -\log \frac{16 k}{9 n^2 \log(n+2)} \\
			& \implies \frac{1}{-\log \frac{16 k}{9 n^2 \log(n+2)}} \leq \frac{1}{- \log V_{n,2}(\nu)} \leq \frac{1}{-\log \frac{16 k}{9 n^2 \log(n+1)}} \\
			& \implies \lim_{n\to \infty} \frac{2 \log n}{-\log \frac{16 k}{9 n^2 \log(n+2)}} \leq \lim_{n \to \infty} \frac{2 \log n}{- \log V_{n,2}(\nu)} \leq \lim_{n \to \infty} \frac{1}{-\log \frac{16 k}{9 n^2 \log(n+1)}} \\
			& \implies 1\leq \lim_{n \to \infty} \frac{2 \log n}{- \log V_{n,2}(\nu)}\leq 1. 
		\end{aligned}
		\]
		This yields the fact that the quantization dimension of the infinite discrete distribution $\nu$ is $1$, which is nonzero. Furthermore,
		\[
		0 = \lim_{n \to \infty} \frac{k}{9 \log(n+2)} \leq \liminf_{n \to \infty} n^2 V_{n,2}(\nu) \leq \limsup_{n \to \infty} n^2 V_{n,2}(\nu) \leq \lim_{n\to \infty} \frac{16k}{9 \log(n+1)} = 0,
		\]
		which implies that the one-dimensional quantization coefficient for $\gn$ is zero. 
	\end{example}
	\begin{remark}
		By Remark~\ref{rem24}, it is known that if the $k$-dimensional quantization coefficient is finite and positive, then $k$ equals the quantization dimension of $\nu.$ By Example \ref{exam23}, we see that the converse is not true. One can also see \cite[Example 6.4]{GL3} in which Graf-Luschgy considered an infinite discrete distribution for which quantization dimension is one, but the quantization coefficient is infinity.
	\end{remark}

	\section*{Statements and Declarations}
	\textbf{Data availability:} Data sharing is not applicable to this article as no data sets were generated or analyzed during the current study.\\
	\textbf{Funding:} The first author thanks IIIT Allahabad (Ministry of Education, India) for financial support through a Senior Research Fellowship. \\
	\textbf{Conflict of interest:} 
	We declare that we do not have any conflict of interest.\\
	\textbf{Author Contributions:} 
	All authors contributed equally to this manuscript.
	
\end{document}